\documentclass{amsart}

\usepackage{hyperref, amssymb, amsmath}

\theoremstyle{plain}
\newtheorem{theorem}{Theorem}
\newtheorem{proposition}[theorem]{Proposition}
\newtheorem{lemma}[theorem]{Lemma}

\theoremstyle{definition}

\newcommand{\CC}{\mathbb{C}}

\newcommand{\A}{\mathcal{A}}
\newcommand{\B}{\mathcal{B}}
\newcommand{\V}{\mathcal{V}}

\newcommand{\diff}{\operatorname{d}}
\newcommand{\euler}{\operatorname{\chi}}
\newcommand{\HH}{\operatorname{H}}
\newcommand{\Ker}{\operatorname{Ker}}
\newcommand{\pr}{\operatorname{pr}}
\newcommand{\Spec}{\operatorname{Spec}}
\newcommand{\T}{\operatorname{T}}

\newcommand{\Ball}{\mathrm{B}}
\newcommand{\eucl}{\mathrm{eucl}}
\newcommand{\Ga}{\mathrm{G}_{\mathrm{a}}}
\newcommand{\id}{\mathrm{id}}
\newcommand{\red}{\mathrm{red}}
\newcommand{\U}{\mathrm{U}}
\newcommand{\var}{\mathrm{var}}

\newcommand{\lra}{\longrightarrow}

\begin{document}

\title[The Effect on Topology of the Action of a Unipotent Group]
{The Effect on Topology of the Action of a Unipotent Group}

\author{Mario Maican}
\address{Institute of Mathematics of the Romanian Academy, Calea Grivitei 21, Bucharest 010702, Romania}

\email{maican@imar.ro}

\begin{abstract}
Assume that two algebraic varieties of finite type over the complex numbers
are related by a morphism whose fibers are precisely the orbits for the action of a unipotent group.
We show that the two varieties have the same topological Euler characteristic.
If they are smooth and the morphism is smooth,
we show that the two varieties have the same cohomology groups.
\end{abstract}


\maketitle

We recall that an \emph{algebraic variety} $X$ over $\CC$ is the set of closed points of a scheme $S$ over $\CC$
together with the restricted sheaf $\mathcal{O}_S^{} |_X^{}$.
Let $G$ be a unipotent linear algebraic group over $\CC$
acting algebraically on a separated algebraic variety $X$ of finite type over $\CC$.
In algebraic geometry, there often arises the need to construct a suitable quotient,
for instance, a geometric quotient, of $X$ modulo $G$.
This problem is not fully understood, in spite of the fact that many particular cases have been investigated.
In this paper we focus on a simpler task.
We assume that we are given a priori a separated variety $Y$ of finite type over $\CC$ parametrizing the $G$-orbits in $X$.
By this we mean the existence of a surjective morphism of algebraic varieties $f \colon X \to Y$,
whose fibers, viewed as reduced subvarieties of $X$, are precisely the $G$-orbits.
We will prove that the Euclidean topologies $X_\eucl^{}$ and $Y_\eucl^{}$ have the same Euler characteristic, see Proposition~\ref{euler}.
If $X$, $Y$ and $f$ are smooth, we will prove that $X_\eucl^{}$ and $Y_\eucl^{}$ have the same cohomology groups, see Proposition~\ref{cohomology}.
Lemmas~\ref{leray}, \ref{lifting} and \ref{contraction} are well-known, but are included here for the sake of completeness.

\begin{lemma}
\label{leray}
Let $X$ and $Y$ be differentiable manifolds that are Hausdorff and have a countable base of open sets.
Let $A$ be an abelian group.
Let $f \colon X \to Y$ be a continuous surjective map.
We assume that $Y$ has a base $\B$ for its topology such that $f^{-1} B$ is acyclic for each $B \in \B$.
Then, for all integers $p \ge 0$, $f$ induces an isomorphism of groups
\[
\HH^p(Y; A) \overset{\simeq}{\lra} \HH^p(X; A).
\]
\end{lemma}

\begin{proof}
Let $\A_X^{}$ and $\A_Y^{}$ denote the sheaves of locally constant functions on $X$, respectively, $Y$, with values in $A$.
By hypothesis, for any integer $q \ge 1$ and $B \in \B$, we have
\[
\HH^q( f^{-1} B; \A_X^{} |_{f^{-1} B}^{}) \simeq \HH^q(f^{-1} B; A) = \{ 0 \}.
\]
It follows that $\mathrm{R}^q f_{*}^{}(\A_X^{}) = \{ 0 \}$ for $q \ge 1$.
The Leray spectral sequence of $\A_X^{}$ degenerates at its second page
\[
\mathrm{E}_2^{pq} = \HH^p(Y; \mathrm{R}^q f_{*}^{}(\A_X^{})).
\]
Thus, for each $p \ge 0$, Leray spectral sequence of $\A_X^{}$ provides an isomorphism
\[
\HH^p(Y; f_{*}^{} \A_X^{}) \overset{\simeq}{\lra} \HH^p(X; \A_X^{}).
\]
But $f^{\sharp} \colon \A_Y^{} \to f_{*}^{} \A_X^{}$ is an isomorphism, hence, for all $p \ge 0$, the natural maps
\[
\HH^p(Y; A) \lra \HH^p(Y; \A_Y^{}) \lra \HH^p(X; \A_X^{}) \lra \HH^p(X; A)
\]
are isomorphisms.
\end{proof}

\noindent
The next lemma concerns a Lie group $G$ and a closed subgroup $H \le G$.
The canonical map $\pi \colon G \to G/H$ is submersive,
so we can define the vertical tangent bundle $\V = \Ker(\diff \pi) \subset \T G$.
Choose a Riemannian scalar product $\langle \, , \rangle_G^{}$ on $\T G$ that is invariant under right translations.
Let $\mathcal{H} \subset \T G$ be the orthogonal complement of $\V$ relative to $\langle \, , \rangle_G^{}$.
Given $h \in H$, both $\V$ and $\langle \, , \rangle_G^{}$ are invariant under right multiplication by $h$,
hence $\mathcal{H}$ is invariant under right multiplication by $h$, as well.

\begin{lemma}
\label{lifting}
Let $\gamma \colon [0, 1] \to G/H$ be a smooth path starting at $\gamma(0) = x$.
Consider a point $z \in \pi^{-1}(x)$.
Then there exists a unique smooth lifting $\tilde{\gamma} \colon [0, 1] \to G$ of $\gamma$,
starting at $\tilde{\gamma}(0) = z$, which is tangent to $\mathcal{H}$ at every point.
\end{lemma}

\begin{proof}
(i) First we establish a local lifting property: for every point $x_0^{} \in G/H$ there exists an open neighbourhood $D$
such that, for every smooth path $\gamma \colon [0, 1] \to D$ starting at $x$, and for any point $z \in \pi^{-1}(x)$,
$\gamma$ admits a unique smooth lifting $\tilde{\gamma} \colon [0, 1] \to G$ starting at $z$, which is tangent to $\mathcal{H}$ at every point.
Indeed, we can choose $D$ so small that there exist a closed submanifold $S \subset \pi^{-1} D$,
which is tangent to $\mathcal{H}$ at every point, and such that $\pi |_S^{} \colon S \to D$ is a diffeomorphism.
We can find $h \in H$ such that $(\pi |_S^{})^{-1}(x)h = z$.
Then $\tilde{\gamma} = ((\pi |_S^{})^{-1} \circ \gamma)h$ is the required lifting of $\gamma$.

\medskip

\noindent
(ii) Consider now a smooth path $\gamma \colon [0, 1] \to G/H$ and a point $z \in \pi^{-1}(\gamma(0))$.
Consider the set $T$ of numbers $t \in (0, 1]$ such that $\gamma |_{[0, t]}$ has a smooth lifting $\tilde{\gamma}_t^{}$
starting at $z$, which is tangent to $\mathcal{H}$ at every point.
From step (i) we know that $T \neq \emptyset$.
We will show that $T$ is both open and closed in $(0, 1]$, so it is the whole interval.
Consider $t_0^{} \in T$. Choose an open neighbourhood $D$ of $\gamma(t_0^{})$ as in step (i).
There is $t \in (t_0^{}, 1]$ such that $\gamma([t_0^{}, t]) \subset D$.
Then $\tilde{\gamma}_{t_0^{}}^{}$ can be extended to a smooth lifting $\tilde{\gamma}_t^{}$,
where $\tilde{\gamma}_t^{} |_{[t_0^{}, t]}$ is the unique smooth lifting of $\gamma |_{[t_0^{}, t]}$, starting at $\tilde{\gamma}_{t_0^{}}^{}(t_0^{})$,
which is tangent to $\mathcal{H}$ at every point.
This shows that $T$ is open.
Assume now that $T$ contains an increasing sequence $\{ t_n^{} \}_{n \ge 1}^{}$ converging to $t \in (0, 1]$.
Choose an open neighbourhood $D$ of $\gamma(t)$ on which the local lifting property holds.
There is an index $n \ge 1$, such that $\gamma([t_n^{}, t]) \subset D$.
Then $\tilde{\gamma}_{t_n^{}}^{}$ can be extended to a lifting $\tilde{\gamma}_t^{}$,
where $\tilde{\gamma}_t^{} |_{[t_n^{}, t]}$ is the unique lifting of $\gamma |_{[t_n^{}, t]}$, starting at $\tilde{\gamma}_{t_n^{}}^{}(t_n^{})$,
which is tangent to $\mathcal{H}$ at every point.
Thus, $t \in T$.
This proves that $T$ is closed.
\end{proof}

\noindent
For the next lemma we recall the group $\U(n, \CC) \subset \mathrm{M}_n(\CC)$ of upper-triangular matrices with $1$ on the diagonal.
Its Lie algebra is $\mathfrak{n}(n, \CC) \subset \mathrm{M}_n(\CC)$, the space of upper-triangular matrices with zero on the diagonal.

\begin{lemma}
\label{contraction}
Let $G$ be a unipotent algebraic group over $\CC$.
There exists a contraction of $G_\eucl^{}$ to its neutral element, which preserves every algebraic subgroup of $G$.
\end{lemma}

\begin{proof}
Since $G$ is isomorphic to an algebraic subgroup of $\U(n, \CC)$, for some $n \ge 1$,
it suffices to prove the lemma in the case when $G = \U(n, \CC)$.
Regarding $\mathfrak{n}(n, \CC)$ as the affine space $\mathbb{A}^{(n - 1)n/2}$, we notice that the maps
\[
\exp \colon \mathfrak{n}(n, \CC) \lra \U(n, \CC) \qquad \text{and} \qquad \log \colon \U(n, \CC) \lra \mathfrak{n}(n, \CC)
\]
are algebraic and inverse to each other, so they are isomorphisms of algebraic varieties.
Consider $u \in \U(n, \CC) \setminus \{ I_n \}$ and $x \in \mathfrak{n}(n, \CC)$ such that $\exp(x) = u$.
We claim that $\langle u \rangle = \{ \exp(c x) \mid c \in \CC \}$ is the algebraic subgroup of $\U(n, \CC)$ generated by $u$.
Indeed, $\langle u \rangle$ is an algebraic subgroup of $\U(n, \CC)$ isomorphic to $\Ga$.
The additive group over $\CC$ has only one non-trivial algebraic subgroup, namely itself.
The maps
\[
E \colon [0, 1] \times \mathfrak{n}(n, \CC) \lra \mathfrak{n}(n, \CC), \qquad E(t, x) = tx,
\]
and
\[
F = \exp \circ E \circ (\id \times \log) \colon [0, 1] \times \U(n, \CC) \lra \U(n, \CC)
\]
are contractions relative to the Euclidean topologies of $\mathfrak{n}(n, \CC)$ and $\U(n, \CC)$.
Since $E$ preserves every complex line in $\mathfrak{n}(n, \CC)$,
it follows that $F$ preserves any subgroup $\langle u \rangle$ of $\U(n, \CC)$.
For an algebraic subgroup $U \le \U(n, \CC)$ we have $U = \bigcup_{u \in U \setminus \{ I_n \}}^{} \langle u \rangle$.
Accordingly, $F$ preserves $U$.
\end{proof}

\begin{proposition}
\label{cohomology}
Let $G$ be a unipotent algebraic group over $\CC$.
Let $f \colon X \to Y$ be a smooth surjective morphism of smooth irreducible separated varieties of finite type over $\CC$.
Assume that $X$ admits an algebraic action of $G$ such that the fibers of $f$ are precisely the $G$-orbits.
Let $A$ be an abelian group.
Then, for all integers $p \ge 0$, $f$ induces an isomorphism of groups
\[
\HH^p(Y_\eucl^{}; A) \overset{\simeq}{\lra} \HH^p(X_\eucl^{}; A).
\]
\end{proposition}

\begin{proof}
All topologies considered in this proof will be the Euclidean topologies.
All tangent spaces and differential maps will be the real tangent spaces and the real differential maps.

\medskip

\noindent
(i) By virtue of Lemma~\ref{leray}, it is enough to show that $Y$ has a base $\B$ of open sets
such that $f^{-1} B$ is acyclic for each $B \in \B$.
We define $\B$ to be the family of open subsets $B \subset Y$ satisfying the following properties:
there exists an open subset $D \subset f^{-1} B$,
there exist balls $B_1^{} = \Ball(o_1^{}, r_1^{})$ and $B_2^{} = \Ball(o_2^{}, r_2^{})$ in Euclidean space, centered at the origin,
and there exist diffeomorphisms $\psi \colon B_1^{} \times B_2^{} \to D$ and $\phi \colon B_1^{} \to B$
such that $\phi^{-1} \circ f \circ \psi = \pr_1^{}$.
Choose arbitrary points $y \in Y$ and $x \in f^{-1}(y)$.
By hypothesis, $f$ is submersive at $x$, so we may apply the constant rank theorem to deduce that
there exist open neighbourhoods $B$ of $y$ and $D$ of $x$ satisfying the above properties.
Any open subset of $B$ that is diffeomorphic to a ball lies in $\B$.
This proves that $\B$ is a base for the Euclidean topology of $Y$.

\medskip

\noindent
(ii) It remains to show that each $f^{-1} B$ is acyclic.
Note that $f |_D$ has a $C^\infty$ section $\xi \colon B \to D$,
where $\psi^{-1} \circ \xi \circ \phi$ is the map $b_1^{} \mapsto (b_1^{}, o_2^{})$.
Consider the surjective $C^\infty$-map
\[
\mu \colon G \times B \lra f^{-1} B, \qquad \mu (g, y) = g . \xi(y).
\]
We will prove that $f^{-1} B$ has a base $\B'$ of open sets such that $\mu^{-1} B'$ is contractible for every $B' \in \B'$.
Invoking again Lemma~\ref{leray}, we will deduce that $f^{-1} B$ and $G \times B$ have the same cohomology.
From Lemma~\ref{contraction} we know that $G$ is contractible, hence $G \times B$ is acyclic.
We will conclude that $f^{-1} B$ is acyclic.

Note that $\mu$ is $G$-equivariant for the action of $G$ on $G \times B$ by left multiplication on the first factor.
By hypothesis, for every $x \in f^{-1} B$ there is $g \in G$ such that $g . x = \xi(f(x))$.
We have reduced the problem to showing that, for every $y \in B$,
$\xi(y)$ has a base of open neighbourhoods $D'$ such that $\mu^{-1} D'$ is contractible.
We will show that any open set of the form $D' = \psi(B_1' \times B_2')$,
where $B_1' \subset B_1^{}$ is a ball centered at $\phi^{-1}(y)$ and $B_2' \subset B_2^{}$ is a ball centered at $o_2^{}$, enjoys this property.
Replacing $B_1'$ with $B_1^{}$ and $B_2'$ with $B_2^{}$, we reduce the problem to proving that $\mu^{-1} D$ is contractible.

\medskip

\noindent
(iii) Given $y \in B$ we denote by $G_y^{}$ the isotropy subgroup of $\xi(y)$.
The universal property of the geometric quotient map $\pi \colon G \to G/G_y^{}$
yields a bijective morphism of algebraic varieties $\upsilon \colon G/G_y^{} \to f^{-1}(y)$ such that $\upsilon(\pi(g)) = g . \xi(y)$.
Here $f^{-1}(y)$ is equipped with the induced reduced structure.
In point of fact, $f^{-1}(y)$ is smooth, because $f$ is smooth.
Invoking Zariski's Main Theorem, we deduce that $\upsilon$ is an isomorphism of algebraic varieties.

\medskip

\noindent
(iv) We claim that $\mu$ is submersive at every point $(g, y)$.
Write $\mu(g, y) = x$.
The restricted map $G \times \{ y \} \to f^{-1}(y)$, $(u, y) \mapsto u . \xi(y)$, is submersive
because the maps $\pi$ and $\upsilon$ from step (iii) are submersive.
The restricted map $\{ g \} \times B \to f^{-1} B$, $(g, v) \mapsto g . \xi(v)$, is a section of $f$,
hence its differential map at $y$ sends $\T_y Y$ surjectively onto a linear complement of $\T_x f^{-1}(y)$ inside $\T_x X$.

We define the vertical tangent bundle $\V = \Ker(\diff \mu) \subset \T(G \times B)$.
We choose a Riemannian scalar product $\langle \, , \rangle_G^{}$ on $\T G$ that is invariant under right-translation.
We choose any Riemannian scalar product $\langle \, , \rangle_B^{}$ on $\T B$.
We consider the Riemannian scalar product
$\langle \, , \rangle_{G \times B}^{} = \pr_1^* \langle \, , \rangle_G^{} + \pr_2^* \langle \, , \rangle_B^{}$ on $\T (G \times B)$.
We consider the orthogonal complement $\mathcal{H} \subset \T (G \times B)$ of $\V$ relative to $\langle \, , \rangle_{G \times B}^{}$.

\medskip

\noindent
(v) Applying Lemma~\ref{lifting} to the the canonical map $\pi \colon G \to G/G_y^{}$,
and taking into account the isomorphism $\upsilon$ from step (iii), we obtain the following lifting property.
For any point $y \in B$, for any smooth path $\gamma \colon [0, 1] \to f^{-1}(y)$ starting at $\gamma(0) = x$,
and for any point $z \in \mu^{-1}(x)$, there exists a unique smooth lifting $\tilde{\gamma} \colon [0, 1] \to G \times B$,
starting at $z$, that is tangent to $\mathcal{H}$ at every point.

\medskip

\noindent
(vi) We construct $C^\infty$ maps
\[
\eta \colon \mu^{-1}(\xi(B)) \times B_2^{} \lra \mu^{-1} D \qquad \text{and} \qquad \theta \colon \mu^{-1} D \lra \mu^{-1}(\xi(B)) \times B_2^{}
\]
as follows. Given $z \in \mu^{-1}(\xi(B))$ and $b_2^{} \in B_2^{}$, write $\mu(z) = x = \xi(y)$, $\phi^{-1}(y) = b_1^{}$,
and consider the smooth path $\gamma \colon [0, 1] \to D$, $\gamma(t) = \psi(b_1^{}, t b_2^{})$, starting at $x$.
Consider the lifting $\tilde{\gamma} \colon [0, 1] \to \mu^{-1} D$ of $\gamma$, starting at $z$, mentioned at step (v).
Set $\eta(z, b_2^{}) = \tilde{\gamma}(1)$.
Conversely, given $z' \in \mu^{-1} D$, write $\mu(z') = x'$, $\psi^{-1}(x') = (b_1', b_2')$,
and consider the smooth path $\delta \colon [0, 1] \to D$, $\delta(t) = \psi(b_1', (1 - t) b_2')$, starting at $x'$.
Consider the lifting $\tilde{\delta} \colon [0, 1] \to \mu^{-1} D$ of $\delta$, starting at $z'$, mentioned at step (v).
Set $\theta(z') = (\tilde{\delta}(1), b_2')$.
By construction, $\eta$ and $\theta$ are inverse to each other and $C^\infty$, so they are diffeomorphisms.

\medskip

\noindent
(vii) According to Lemma~\ref{contraction}, there exists a homotopy $F \colon [0, 1] \times G \to G$
between $\id_G^{}$ and the constant map $g \mapsto e$.
Here $e$ is the neutral element of $G$.
Moreover, if $H$ is an algebraic subgroup of $G$, then $F(t, - )$ maps $H$ to $H$, for all $t \in [0, 1]$.
Notice that
\[
\mu^{-1}(\xi(B)) = \{ (g, y) \in G \times B \mid g \in G_y^{} \}.
\]
Each $G_y^{}$ is an algebraic subgroup of $G$, hence $F \times \id_B^{}$ restricts to a continuous map
\[
[0, 1] \times \mu^{-1}(\xi(B)) \lra \mu^{-1}(\xi(B)),
\]
providing a homotopy between the identity map of $\mu^{-1}(\xi(B))$ and the map $(g, y) \mapsto (e, y)$.
Since $B$ is contractible, we deduce that $\mu^{-1}(\xi(B))$ is contractible.
Since $B_2^{}$ is contractible, and in view of the diffeomorphism $\eta$ from step (vi), we find that $\mu^{-1} D$ is contractible, as well.
As noted at step (ii), this concludes the proof of the proposition.
\end{proof}

\noindent
A finite family $\{ Y_i^{} \}_{i \in I}^{}$ of locally closed subvarieties of $Y$ is called a \emph{decomposition} of $Y$
if the canonical morphism $\bigsqcup_{i \in I}^{} Y_i^{} \to Y$ is a bijection.

\begin{lemma}
\label{flattening}
Let $f \colon X \to Y$ be a surjective morphism of algebraic varieties of finite type over $\CC$.
Assume that $X$ admits an algebraic action of a connected algebraic group $G$ such that,
for every $y \in Y$, $f^{-1}(y)$, regarded as a subset of $X$, is a $G$-orbit.
Then $Y$ admits a decomposition $\{ Y_i^{} \}_{i \in I}^{}$ into finitely many integral locally closed subvarieties
such that all morphisms $f_i^{} \colon X_i^{} \to Y_i^{}$ obtained from $f$ by base extension are flat
and such that all $X_i^{}$ are irreducible.
\end{lemma}

\begin{proof}
We first perform a flattening stratification of $X$.
There exists a finite decomposition $\{ Y_i^{} \}_{i \in I}^{}$ of $Y$ into locally closed reduced subvarieties
such that all morphisms $f_i^{} \colon X_i^{} \to Y_i^{}$ induced by base extension are flat.
Consult \cite[Lemma 2.1.6]{huybrechts-lehn}.
It is enough to prove the lemma for each morphism $f_i^{}$
relative to the induced action of $G$ on $X_i^{} = X \times_Y^{} Y_i^{}$ by multiplication on the first component.
Thus, we may assume a priori that $f$ is flat and that $Y$ is reduced.

We perform induction on $\dim X$.
If $\dim X = 0$, then $Y = \{ y_1^{}, \dots, y_m^{} \}$ is a finite set of points
and the required decomposition is obtained by setting $Y_i^{} = \{ y_i^{} \}$, for $1 \le i \le m$.
Assume now that $\dim X > 0$.
Let $Z_1^{}, \ldots, Z_n^{}$ be the irreducible components of $X$.
Denote $U_j^{} = Z_j^{} \setminus \bigcup_{k \neq j}^{} Z_k^{}$.
Multiplication by an element of $G$ permutes the sets $U_j^{}$, hence $G . U_j^{} \subset U_1^{} \cup \dots \cup U_n^{}$.
But $G . U_j^{}$ is irreducible, hence $G . U_j^{} = U_j^{}$.
It follows that $U_j^{} = f^{-1} V_j^{}$, where $V_j^{} = f(U_j^{})$.
A flat morphism of schemes of finite type over a field is open, hence $V_j^{}$ is open.
Being irreducible and reduced, each $V_j^{}$ is an open integral subvariety of $Y$.
Consider the closed subvariety $Y' = Y \setminus \bigcup_{1 \le j \le n}^{} V_j^{}$ equipped with the induced reduced structure.
Applying the induction hypothesis to the morphism $f' \colon X' \to Y'$ obtained from $f$ by base extension,
yields a decomposition $\{ Y'_i \}_{i \in I}^{}$ of $Y'$, as in the lemma.
Then $\{ V_j^{} \}_{1 \le j \le n}^{} \cup \{ Y'_i \}_{i \in I}^{}$ is the required decomposition of $Y$.
\end{proof}

\noindent
For a morphism $f \colon X \to Y$ of algebraic varieties over $\CC$,
we denote by $f_\red^{} \colon X_\red^{} \to Y_\red^{}$ the induced morphism of reduced varieties.

\begin{lemma}
\label{smoothening}
Under the assumptions of Lemma~\ref{flattening}, we claim that $Y$ admits a decomposition $\{ Y_i^{} \}_{i \in I}^{}$
into finitely many integral smooth locally closed subvarieties satisfying the following properties:
\begin{enumerate}
\item[(i)]
all morphisms $f_i^{} \colon X_i^{} \to Y_i^{}$ obtained from $f$ by base extension are flat;
\item[(ii)]
all $X_i^{}$ are irreducible;
\item[(iii)]
all $(X_i)_\red^{}$ are smooth;
\item[(iv)]
all morphisms $(f_i^{})_\red^{} \colon (X_i^{})_\red^{} \to Y_i^{}$ are smooth.
\end{enumerate}
\end{lemma}

\begin{proof}
We perform induction on $\dim X$.
If $\dim X = 0$, then $Y_\red^{} = \{ y_1^{}, \ldots, y_m^{} \}$ is a finite set of points
and the required decomposition is obtained by setting $Y_i^{} = \{ y_i^{} \}$, for $1 \le i \le m$.
Assume now that $\dim X > 0$.
We consider a decomposition $\{ Y_i^{} \}_{i \in I}^{}$ of $Y$ as in Lemma~\ref{flattening}.
Let $U_i^{} \subset (X_i^{})_\red^{}$ be the set of regular points.
Since $U_i^{}$ is $G$-invariant, $U_i^{} = f^{-1} V_i^{}$, where $V_i^{} = f(U_i^{})$.
A flat morphism of schemes of finite type over a field is open, hence $V_i^{}$ is open in $Y_i^{}$.
Choose an open non-empty and non-singular subset $V_i' \subset V_i$ and denote $U_i' = f_i^{-1} V_i'$.
Applying \cite[Lemma III.10.5]{hartshorne} to the surjective morphism $(f_i^{})_\red^{} \colon U_i' \to V_i'$,
we obtain an open subset $U_i'' \subset U_i'$ such that the restricted morphism $(f_i^{})_\red^{} \colon U_i'' \to V_i'$ is smooth.
According to \cite[Proposition III.10.4]{hartshorne}, for each $x \in U_i''$, the complex differential map
\[
\diff (f_i^{})_\red^{}(x) \colon \T_x^{} (X_i^{})_\red^{} \lra \T_{f(x)}^{} Y_i^{}
\]
is surjective.
The same is true for each $x \in G . U_i''$, because $(f_i^{})_\red^{}$ is $G$-equivariant.
Consider the open subset $W_i^{} = f_i^{}(U_i'') \subset Y_i^{}$.
By construction, $(f_i^{-1} W_i^{})_\red^{}$ is smooth and the restricted morphism
$(f_i^{})_\red^{} \colon (f_i^{-1} W_i^{})_\red^{} \to W_i^{}$ is smooth.

Consider the closed subvariety $Y_i' = Y_i^{} \setminus W_i^{}$ equipped with the induced reduced structure.
Let $f_i' \colon X_i' \to Y_i'$ be the morphism obtained from $f_i$ by base extension.
Applying the induction hypothesis to $f_i'$, we obtain a decomposition $\{ Y_{ij}^{} \}_{j \in J_i^{}}^{}$ of $Y_i'$, as in the lemma.
In conclusion, $\{ W_i^{} \}_{i \in I}^{} \cup \{ Y_{ij}^{} \mid i \in I, \, j \in J_i^{} \}$
constitutes a finite decomposition of $Y$ into integral smooth locally closed subvarieties, satisfying the properties from the lemma.
\end{proof}

\noindent
For an algebraic variety $X$ of finite type over $\CC$,
we denote by $\euler(X)$ the topological Euler characteristic of the Euclidean topology of $X$.

\begin{proposition}
\label{euler}
Let $G$ be a unipotent algebraic group over $\CC$.
Let $f \colon X \to Y$ be a surjective morphism of separated algebraic varieties of finite type over $\CC$.
Assume that $X$ admits an algebraic action of $G$ such that,
for every $y \in Y$, $f^{-1}(y)$, regarded as a subset of $X$, is a $G$-orbit.
Then $\euler(X) = \euler(Y)$.
\end{proposition}

\begin{proof}
Let $\{ Y_i^{} \}_{i \in I}^{}$ be a finite decomposition of $Y$ into locally closed subvarieties, as in Lemma~\ref{smoothening}.
Note that $\{ X_i^{} = X \times_Y^{} Y_i^{} \}_{i \in I}^{}$ is a finite decomposition of $X$ into locally closed subvarieties.
The morphisms $(f_i^{})_\red^{} \colon (X_i^{})_\red^{} \to Y_i^{}$ satisfy the hypotheses of Proposition~\ref{cohomology},
hence the Euclidean topologies of $X_i^{}$ and $Y_i^{}$ have the same cohomology groups.
Accordingly, $\euler(X_i^{}) = \euler(Y_i^{})$.
Using the additivity of $\euler$, we calculate:
\[
\euler(X) = \sum_{i \in I} \euler(X_i^{}) = \sum_{i \in I} \euler(Y_i^{}) = \euler(Y). \qedhere
\]
\end{proof}

\noindent
The above proposition is reminiscent of the method of $U$-invariants, which we summarize below.
Let $X = \Spec R$ be an affine scheme over $\CC$.
Let $G$ be a reductive connected algebraic group over $\CC$ acting algebraically on $X$.
Let $U$ be a maximal unipotent subgroup of $G$.
Let $Y = \Spec R^U$.
Then:
\begin{enumerate}
\item[(i)]
$X$ is of finite type over $\CC$ if and only if $Y$ is of finite type over $\CC$;
\item[(ii)]
$X$ is reduced if and only if $Y$ is reduced;
\item[(iii)]
$X$ is integral if and only if $Y$ is integral;
\item[(iv)]
$X$ is normal if and only if $Y$ is normal.
\end{enumerate}
We refer to \cite[Theorem 1.3.1]{hyun}.
We denote by $f \colon X_\var^{} \to Y_\var^{}$ the induced morphism of algebraic varieties over $\CC$.
By virtue of Proposition~\ref{euler}, we can add a fifth property to the above:
\begin{enumerate}
\item[(v)]
Assume, in addition, that $X$ is of finite type over $\CC$, that $f$ is surjective and that the fibers of $f$ are precisely the $U$-orbits.
Then $\euler(X) = \euler(Y)$.
\end{enumerate}


\begin{thebibliography}{99}

\bibitem{hartshorne}
R. Hartshorne.
\emph{Algebraic Geometry},
Springer-Verlag, 1977.

\bibitem{huybrechts-lehn}
D. Huybrechts, M. Lehn.
\emph{The Geometry of Moduli Spaces of Sheaves},
second edition, Cambridge University Press, 2010.

\bibitem{hyun} Y. Hyun.
\emph{On Affine Embeddings of Reductive Groups},
Ph.D. dissertation, Massachusetts Institute of Technology, 2011.

\end{thebibliography}
\end{document}